\documentclass{amsart}

\usepackage{amsfonts, amssymb}
\usepackage{amsmath}
\usepackage{amsthm}

\numberwithin{equation}{section}

\newtheorem{theorem}{Theorem}[section]
\newtheorem{lemma}[theorem]{Lemma}
\newtheorem{example}[theorem]{Example}
\newtheorem{proposition}[theorem]{Proposition}
\newtheorem{corollary}[theorem]{Corollary}

\newcommand{\C}{\mathbb{C}}
\newcommand{\N}{\mathbb{N}}
\newcommand{\Z}{\mathbb{Z}}
\renewcommand{\P}{\mathbb{P}}

\newcommand{\ord}{\textrm{ord}}
\newcommand{\rad}{\textrm{rad}}

\begin{document}

\title{Mason's theorem with a difference radical}

\author{Katsuya Ishizaki}
\address{The Open University of Japan, 2-11 Wakaba, Mihama-ku, Chiba, 261-8586 JAPAN}
\email{ishizaki@ouj.ac.jp}

\author{Risto Korhonen}
\address{Department of Physics and Mathematics, University of Eastern Finland, P.O. Box 111,
FI-80101 Joensuu, Finland}
\email{risto.korhonen@uef.fi}

\author{Nan Li}

\address{University of Jinan, School of Mathematical Sciences, Jinan, Shandong, 250022, P.R. China}
\email{sms\_lin@ujn.edu.cn}

\author{Kazuya Tohge}
\address{College of Science and Engineering, Kanazawa University, Kakuma-machi, Kanazawa, 920-1192, Japan}
\email{tohge@se.kanazawa-u.ac.jp}

\thanks{The first author was supported by the discretionary budget (2017) of the President of the Open University of Japan. The second author wishes to acknowledge partial support by the Academy of Finland grant (\#286877). The third author was supported in part by the grant 11626112 from the NSFC Tianyuan Mathematics Youth Fund,
the grant ZR2016AQ20, the NNSF of China No.11371225, and the grant XBS1630 from the Fund of Doctoral Program Research of University of Jinan.
The work of the fourth author was supported by JSPS KAKENHI Grant Number JP16K05194.}


\date{\today}

\commby{}

\begin{abstract}
Differential calculus is not a unique way to observe polynomial equations such as $a+b=c$.
We propose a way of applying difference calculus to estimate multiplicities of the roots of the polynomials $a$, $b$ and $c$ satisfying the equation above.
Then a difference $abc$ theorem for polynomials is proved using a new notion of a radical of a polynomial. Two results on the non-existence of polynomial solutions to difference Fermat type functional equations are given as applications. We also introduce a truncated second main theorem for differences, and use it to consider difference Fermat type equations with transcendental entire solutions.
\end{abstract}

\maketitle



\section{Introduction}

Mason's theorem states that if relatively prime polynomials $a$, $b$ and $c$, not all of them identically zero, satisfy
    \begin{equation*}
    a+b=c,
    \end{equation*}
then $\deg c\leq \deg\rad(abc)-1$, where the radical $\rad(abc)$ is the product of distinct linear factors of $abc$ \cite{mason:84,stothers:81}, see also \cite{snyder:00}. An elementary application of Mason's theorem is that if $x$, $y$ and $z$ are non-trivial relatively prime polynomials satisfying
    \begin{equation}\label{polyFermat}
    x^n + y^n = z^n,
    \end{equation}
where $n\in\N$, then $n\leq 2$. Mason's theorem is a counterpart of the $abc$ conjecture in number theory, while its consequence described above is Fermat's last theorem for polynomials (see, e.g., \cite{laeng:99,lang:90}).

Fermat type functional equations, such as \eqref{polyFermat} and its generalizations have been studied over many function fields \cite{dyakonov:12,hayman:84,huy:02}  (see also, e.g., \cite{gundersenh:04} and the references therein). For instance, if
    \begin{equation}\label{genFermat}
    f_1^n + f_2^n +  \cdots + f_m^n = 1
    \end{equation}
has a solution consisting of $m$ polynomials $f_1,f_2,\ldots,f_m$, then $n\leq m^2 - m -1$. For rational, entire and meromorphic solutions the corresponding bounds are $n\leq m^2-2$, $n\leq m^2 - m$ and $n\leq m^2 -1$, respectively \cite{hayman:84}. Hayman~\cite{hayman:14} calls the problem of finding the smallest $m=G_0(n)$ for which a solution of \eqref{genFermat} exists as the Super-Fermat problem. A difference analogue of \eqref{genFermat} was studied by the third author~\cite{li:15}, who obtained similar bounds for a difference counterpart of \eqref{genFermat} under certain conditions on the value distribution of solutions.

The purpose of this paper is to introduce a difference counterpart of the radical, and to use it to prove a difference analogue of Mason's theorem, as well as a truncated version of the difference second main theorem for holomorphic curves. As applications we prove two results on the non-existence of polynomial solutions to difference Fermat type equations, and two non-existence results on difference Fermat type equations with transcendental entire solutions.

\section{Difference radical}

Let $p\not\equiv 0$ be a polynomial in $\C[z]$, and let $\kappa\in \C\setminus\{0\}$. We define the $\kappa$-{\it difference radical} $\mathrm{r\tilde{a}d}_{\kappa}(p)$ of $p$ as
$$
\mathrm{r\tilde{a}d}_{\kappa}\bigl(p(z)\bigr)=\prod_{w\in\mathbb{C}}(z-w)^{d_{\kappa}(w)},
$$
where
$$
d_{\kappa}(w)=\ord_w(p)-\min\{\ord_w(p), \ord_{w+\kappa}(p) \}
$$
with $\ord_w (p)\geq 0$ being the order of zero of the polynomial $p$ at $w\in \C$. This corresponds to the way to define the usual radical $\rad\, p$ as
$$
\rad\, p(z) = \prod_{w\in\mathbb{C}}(z-w)^{d(w)},
$$
where
$$
d(w)=\ord_w(p)-\min \{ \ord_w(p), \ord_w(p') \} \in \{0,1\}.
$$
Now, by defining $\tilde{n}_{\kappa}(p)=\deg\mathrm{r\tilde{a}d}_{\kappa}(p)$, it follows that $\tilde n_\kappa(p)$ is the number of zeros of $p$ appearing non-periodically with respect to the constant $\kappa$, where the multiplicities of the zeros are taken into account. In other words,
    \begin{equation}\label{tilden}
    \tilde n_\kappa(p) = \sum_{w\in\C} \left(\ord_w (p) - \min\{\ord_{w} (p),\ord_{w+\kappa}(p)\}\right).
    \end{equation}
For example, if $p$ has zeros of order $2$, $1$ and $3$ at $z_0$, $z_0+1$ and $z_0+2$, respectively, and no zero at $z_0+3$, then the zero of $p$ at $z_0$ is counted once in $\tilde n_1(p)$ and the zero at $z_0+2$ three times in $\tilde n_1(p)$, while the zero at $z_0+1$ is not counted in $\tilde n_1(p)$.

In addition, we define $\Delta_\kappa p = p(z+\kappa)-p(z)$, and use the notation $\gcd(p,q)$ to denote the greatest common divisor of $p$ and $q$ over $\C[z]$. 

\begin{lemma}\label{difference_lemma}
Let $p\not\equiv0$ be a polynomial in $\C[z]$. Then,
    \begin{equation*} \label{representation}
    p=\gcd(p, \Delta_{\kappa}p) \cdot \mathrm{r\tilde{a}d}_{\kappa}(p)
    \end{equation*}
and therefore
    \begin{equation*}
    \deg p  = \deg \gcd(p,\Delta_\kappa p) + \tilde n_\kappa(p).
    \end{equation*}
\end{lemma}

\begin{proof}
We may write $p$ in the form
    \begin{equation}\label{product_form}
    p(z)= \gamma\prod_{i=1}^m \prod_{j=0}^{l_i} (z-\beta_i+j\kappa),
    \end{equation}
where $\gamma\in \C$ and $l_i\in \N\cup\{0\}$. Note that the roots of $p$ are repeated in \eqref{product_form} the number of times according to their multiplicity, so the case $\beta_i=\beta_k$, $i\not=k$, is allowed. More precisely, for a zero of $p(z)$, if $\ord_\beta (p)>\ord_{\beta+\kappa}(p)$, then $\beta$ is entered $\ord_\beta (p)-\ord_{\beta+\kappa}(p)$ times as one of the `$\{\beta_i\}$' in \eqref{product_form}.  If $\ord_\beta (p)\leq\ord_{\beta+\kappa}(p)$, then $\beta$ is not entered as one of the `$\{\beta_i\}$' in \eqref{product_form}. Moreover, we may assume in \eqref{product_form} that $\beta_s\ne\beta_t-(l_{t}+1)\kappa$ for any $s$, $t=1, 2, \dots, m$, since otherwise, we can combine two products as
\begin{equation*}
\prod_{j=0}^{l_t}(z-\beta_t+j\kappa)\cdot \prod_{j=0}^{l_s}(z-\beta_s+j\kappa)=\prod_{j=0}^{l_t+l_s+1}(z-\beta_t+j\kappa).
\end{equation*}
Now, by \eqref{product_form}, the difference radical satisfies the simple representation
$$
\mathrm{r\tilde{a}d}_{\kappa}(p(z))=\prod_{i=1}^{m} (z-\beta_i).
$$
In fact, for each $i\in\{1, \ldots , m\}$ we have
\[
d_{\kappa}(\beta_i-j\kappa) = \left\{
\begin{array}{ll}
1-\min(1,1)=0 & (1\leq j \leq \ell_i), \\
1-\min(1,0)=1 & (j=0).
\end{array}
\right.
\]
From \eqref{product_form}, we have
    \begin{equation*}\label{product_form2}
    p(z+\kappa)= \gamma\prod_{i=1}^m \prod_{j=0}^{l_i} (z-\beta_i+(j+1)\kappa),
    \end{equation*}
and so
    \begin{equation*}
    p(z) =   \gcd \left(p(z), p(z+\kappa)\right) \cdot \mathrm{r\tilde{a}d}_{\kappa}\bigl(p(z)\bigr).
    \end{equation*}
Since
\begin{equation*}
 \gcd \left(p(z), p(z+\kappa)\right)=\gcd \left(p(z), p(z+\kappa)-p(z)\right)
 =\gcd\left(p(z),\Delta_\kappa p(z)\right),
\end{equation*}
it follows that
\begin{equation*}
  p(z) =   \gcd\left(p(z),\Delta_\kappa p(z)\right)\cdot \mathrm{r\tilde{a}d}_{\kappa}\bigl(p(z)\bigr).
\end{equation*}
Therefore,
\begin{equation*}
  \deg p  = \deg \gcd(p,\Delta_\kappa p) + \deg\mathrm{r\tilde{a}d}_{\kappa}(p) = \deg \gcd(p,\Delta_\kappa p) + \tilde n_\kappa(p).
\end{equation*}
\end{proof}

In what follows we denote $\bar{n}(p)=\deg\mathrm{rad}\, p$ for the number of all the distinct roots of $p(z)$. Then, as Laeng~\cite{laeng:99} observed, we obtain the following properties for~$\bar{n}$:
\begin{enumerate}
\item $\bar{n}(p)\leq \deg p$ for any $p(z)\in\mathbb{C}[z]$;
\item $\bar{n}(p^m)= \bar{n}(p)$ for any $p(z)\in\mathbb{C}[z]$ and $m\in\mathbb{N}$;
\item $\bar{n}(pq)\leq \bar{n}(p)+\bar{n}(q)$ for any $p(z), q(z)\in\mathbb{C}[z]$, where the equality holds exactly when $p(z)$ and $q(z)$ are relatively prime.
\end{enumerate}

Of the $\kappa$-difference analogue $\tilde{n}_{\kappa}$, those properties change slightly but significantly. In fact, our definition \eqref{tilden} of $\tilde{n}_{\kappa}(p)$
implies:
\begin{enumerate}
\item[($\tilde1$)] $\tilde{n}_{\kappa}(p)\leq \deg p$ for any $p(z)\in\mathbb{C}[z]$;
\item[($\tilde2$)] $\tilde{n}_{\kappa}(p^m)= m\cdot\tilde{n}_{\kappa}(p)$ for any $p(z)\in\mathbb{C}[z]$ and $m\in\mathbb{N}$;
\item[($\tilde3$)] $\tilde{n}_{\kappa}(pq)\leq \tilde{n}_{\kappa}(p)+\tilde{n}_{\kappa}(q)$ for any $p(z), q(z)\in\mathbb{C}[z]$, where the equality holds exactly when both $\mathrm{r\tilde{a}d}_{-\kappa}(p(z+\kappa))$ and $\mathrm{r\tilde{a}d}_{\kappa}(q(z))$, as well as $\mathrm{r\tilde{a}d}_{\kappa}(p(z))$ and $\mathrm{r\tilde{a}d}_{-\kappa}(q(z+\kappa))$ are relatively prime.
\end{enumerate}

%

%

\section{Difference Analogue of Mason's theorem}

The following theorem is a difference analogue of Mason's theorem, or in other words, a difference $abc$ theorem for polynomials.

\begin{theorem}\label{differenceMason}
Let $a$, $b$ and $c$ be relatively prime polynomials in $\C[z]$ such that
    \begin{equation}\label{abc}
    a+b=c
    \end{equation}
and such that $a$, $b$ and $c$ are not all constant. Then, 
    \begin{equation*}
    \max\{\deg a,\deg b,\deg c\} \leq \tilde n_\kappa(a) + \tilde n_\kappa(b) + \tilde n_\kappa(c) - 1,
    \end{equation*}
where $\kappa\in\C\setminus\{0\}$.
\end{theorem}

\begin{proof}
Without loss of generality we may assume that $\max\{\deg a,\deg b,\deg c\}=\deg c$. From \eqref{abc} it follows that
    \begin{equation*}
    \Delta_\kappa a + \Delta_\kappa b = \Delta_\kappa c.
    \end{equation*}
Thus
    \begin{equation}\label{abc_eq1}
    a\Delta_\kappa a + a\Delta_\kappa b = a\Delta_\kappa c
    \end{equation}
and
    \begin{equation}\label{abc_eq2}
    a\Delta_\kappa a + b\Delta_\kappa a = c\Delta_\kappa a.
    \end{equation}
By subtracting \eqref{abc_eq2} from \eqref{abc_eq1}, we have
    \begin{equation*}\label{abc_eq3}
     a\Delta_\kappa b - b\Delta_\kappa a  = a\Delta_\kappa c -c\Delta_\kappa a,
    \end{equation*}
and so $\gcd(a,\Delta_\kappa a)$, $\gcd(b,\Delta_\kappa b)$ and $\gcd(c,\Delta_\kappa c)$ are all factors of $a\Delta_\kappa b - b\Delta_\kappa a$. Since $a$, $b$ and $c$ are relatively prime, it follows that also $\gcd(a,\Delta_\kappa a)$, $\gcd(b,\Delta_\kappa b)$ and $\gcd(c,\Delta_\kappa c)$ are relatively prime. Therefore,
    \begin{equation*}
    \gcd(a,\Delta_\kappa a)\gcd(b,\Delta_\kappa b)\gcd(c,\Delta_\kappa c)
    \end{equation*}
is a factor of $a\Delta_\kappa b - b\Delta_\kappa a$, which implies that
    \begin{equation}\label{degineq}
    \deg\gcd(a,\Delta_\kappa a) + \deg\gcd(b,\Delta_\kappa b) + \deg\gcd(c,\Delta_\kappa c) \leq \deg a + \deg b - 1
    \end{equation}
provided that $a\Delta_\kappa b - b\Delta_\kappa a\not = 0$. But if
    \begin{equation}\label{at}
    a\Delta_\kappa b - b\Delta_\kappa a= 0,
    \end{equation}
then $a\Delta_\kappa b = b\Delta_\kappa a$, and so $a$ is a factor of $b\Delta_\kappa a$. Since $a$ and $b$ have no common factors, it follows that $a$ is a factor of $\Delta_\kappa a$. This is only possible if $\Delta_\kappa a=0$. Similarly, under the assumption \eqref{at} it follows that $\Delta_\kappa b=0$ and $\Delta_\kappa c=0$, which contradicts the assumption of the theorem. Hence, \eqref{at} cannot hold and \eqref{degineq} is valid. By adding $\deg c$ to both sides of \eqref{degineq} and reorganizing the terms, we have
    \begin{equation*}
    \begin{split}
    \deg c &\leq \deg a - \deg\gcd(a,\Delta_\kappa a) + \deg b - \deg\gcd(b,\Delta_\kappa b) \\
    &\quad + \deg c - \deg\gcd(c,\Delta_\kappa c) - 1.
    \end{split}
    \end{equation*}
The assertion follows by Lemma~\ref{difference_lemma}.
\end{proof}

\begin{example}\label{ex32}
We can see that the assertion of Theorem~\ref{differenceMason} is sharp by the example $a(z)=(z+\alpha)(z+\alpha+\kappa)$, $b(z)=-(z+\beta)(z+\beta+\kappa)$, and $c(z)=2(\alpha-\beta)(z+(\alpha+\beta+\kappa)/2)$, where $\alpha,\beta\in\C$ such that $\beta\not=\alpha\not=\beta\pm\kappa$. Namely, then $a$, $b$ and $c$ are relatively prime polynomials in $\C[z]$
such that $a+b=c$, and such that none of the differences $\Delta_{\kappa}a$,  $\Delta_{\kappa}b$ and $\Delta_{\kappa}c$ is identically zero. In addition, $\max \{\deg a,\deg b,\deg c\}=2$, $\widetilde{n}_{\kappa}(a)=1$,
$\widetilde{n}_{\kappa}(b)=1$, $\widetilde{n}_{\kappa}(c)=1$ and $\widetilde{n}_{\kappa}(a)+\widetilde{n}_{\kappa}(b)+\widetilde{n}_{\kappa}(c)-1=2$.
\end{example}

Example~\ref{ex32} shows that Theorem~\ref{differenceMason} is sharp when $\max\{\deg a, \deg b,\deg c\}=2$.
The following example demonstrates the sharpness of Theorem~\ref{differenceMason} for the case $\max\{\deg a, \deg b,\deg c\}=4$ and $\kappa=1$.

\begin{example}
Set $\displaystyle \nu=\frac{1+\sqrt{3}i}{2}$ noting that $\nu^3=-1$, and set $\alpha=1-\nu$, $\beta=\nu$. Define
\begin{align*}
a(z)&=A(z+\alpha)^2(z+\alpha+1)^2,\quad b(z)=-A(z+\beta)^2(z+\beta+1)^2,\\
c(z)&=z(z+1)(z+2)\quad \text{with $\displaystyle A=\frac{i}{4\sqrt{3}}$}.
\end{align*}
Then $a$, $b$ and $c$ satisfy $\eqref{abc}$, and $\max\{\deg a, \deg b,\deg c\}=4$,
$\tilde{n}_1(a)=2$, $\tilde{n}_1(b)=2$, $\tilde{n}_1(c)=1$.
\end{example}

The following result extends Theorem~\ref{differenceMason} for $m+1$ polynomials.

\begin{theorem}\label{differenceMason_m+1}
Let $a_1,\ldots,a_{m+1}$ be relatively prime polynomials in $\C[z]$ such that
    \begin{equation}\label{a1am}
    a_1+ \ldots + a_m=a_{m+1},
    \end{equation}
and such that $a_1,\ldots,a_{m}$ are linearly independent over $\C$. Then,
    \begin{equation}\label{diffMasonineq}
    \max_{1\leq i\leq m+1}\{\deg a_i\} \leq \sum_{i=1}^{m+1}\tilde{n}_{\kappa}^{[m-1]}(a_i)-\frac{1}{2}m(m-1),
    \end{equation}
where we denote
\begin{equation} \label{eqn:n^m}
\tilde{n}_{\kappa}^{[m-1]}(a_i)=\deg\mathrm{r\tilde{a}d}_{\kappa}^{[m-1]}(a_i)
=\sum_{w\in\C} \Bigl(\mathrm{ord}_{w} (a_i)-\min_{0\leq j\leq m-1}\bigl\{\mathrm{ord}_{w+j\kappa} (a_i)\bigr\}\Bigr)
\end{equation}
and $\kappa\in\C\setminus\{0\}$.
\end{theorem}

\begin{proof}
Now we consider the Casoratian ${\mathcal C}_{\kappa}(z)\not\equiv0$ of $a_1(z), \ldots , a_m(z)$.
Let $z_{0}$ be a zero of some $a_i(z)$ with $1\leq i \leq m+1$.
Then ${\mathcal C}_{\kappa}(z)$ has also a zero at $z=z_{0}$ of multiplicity not smaller than
$$
\min_{0\leq j\leq m-1} \{ \ord_{z_{0}} (a_i(z+j\kappa)) \}.
$$
Therefore, under the assumption of the relative primeness we see that
$$
q(z):=\prod_{i=1}^{m+1} \gcd\Bigl(a_i(z), a_i(z+\kappa), \ldots , a_i\bigl(z+(m-1)\kappa\bigr) \Bigr)
$$
divides ${\mathcal C}_{\kappa}(z)$, so that there exists a polynomial $p(z)\in \mathbb{C}[z]$ satisfying
${\mathcal C}_{\kappa}(z)=p(z)q(z)$.
Note that the degree of $q(z)$ is not less than
$$
\sum_{i=1}^{m+1} \sum_{w\in\mathbb{C}} \min_{0\leq j\leq m-1} \bigl\{ \ord_w (a_i(z+j\kappa)) \bigr\} =
\sum_{i=1}^{m+1}\Bigl[\sum_{w\in\mathbb{C}}\ord_w(a_i)-\tilde{n}_{\kappa}^{[m-1]}(a_i) \Bigr]
$$
by means of the notation $(\ref{eqn:n^m})$.

On the other hand, the degree of ${\mathcal C}_{\kappa}(z)$ is never beyond any sum of distinct $m$ of the $\deg a_i(z)$ $(1\leq i\leq m+1)$ minus $\sum_{\ell=0}^{m-1}\ell=m(m-1)/2$ as the sum of $\deg\bigl(\Delta_{\kappa}^{\ell}a_{i_{\nu}}\bigr)$ for the mutually distinct $m$ integers $i_{\nu}\in\{1, \ldots , m, m+1\}$.
Hence we obtain
\begin{eqnarray*}
\min_{1\leq k\leq m+1} \sum_{1\leq i\leq m+1, i\neq k} \deg a_i -\frac{1}{2}m(m-1)
&\geq&
\sum_{i=1}^{m+1}\Bigl[\sum_{w\in\mathbb{C}}\ord_w(a_i)-\tilde{n}_{\kappa}^{[m-1]}(a_i) \Bigr]\\
&=&
\sum_{i=1}^{m+1}\deg a_i -\sum_{i=1}^{m+1} \tilde{n}_{\kappa}^{[m-1]}(a_i).
 \end{eqnarray*}
This implies our desired estimate
$$
 \max_{1\leq i\leq m+1}\{\deg a_i\} \leq \sum_{i=1}^{m+1}\tilde{n}_{\kappa}^{[m-1]}(a_i)-\frac{1}{2}m(m-1).
$$
\end{proof}

Note that $\tilde{n}_{\kappa}^{[m-1]}(a_i)$ in the above estimate cannot be replaced by $\tilde{n}_{\kappa}(a_i)=\deg\mathrm{r\tilde{a}d}_{\kappa}(a_i)$ to obtain
\begin{equation} \label{eqn:not}
 \max_{1\leq i\leq m+1}\{\deg a_i\} \leq \sum_{i=1}^{m+1}\deg\mathrm{r\tilde{a}d}_{\kappa}(a_i)-\frac{1}{2}m(m-1).
\end{equation}
By definition
$$
\tilde{n}_{\kappa}^{[m-1]}(a_i)-\tilde{n}_{\kappa}(a_i)
=\sum_{w\in\mathbb{C}}\Bigl[\min\{\ord_w (a_i), \ord_{w+\kappa} (a_i)\bigr\}-\min_{0\leq j \leq m-1}\bigl\{\ord_{w+j\kappa}(a_i)\Bigr]
$$
is always non-negative so that $\tilde{n}_{\kappa}^{[m-1]}(a_i)\geq \tilde{n}_{\kappa}(a_i)$.
Example~\ref{example_crude} below shows that $(\ref{eqn:not})$ does not hold in general when $m>2$.
If one wished to use the radicals $\mathrm{r\tilde{a}d}_{\kappa}(a_i)$ in \eqref{diffMasonineq}, it is possible to use such estimates as
\begin{eqnarray*}
\tilde{n}_{\kappa}^{[m-1]}(a_i)&=& \sum_{w\in\C}\Bigl\{\ord_w(a_i)-\min\bigl\{\ord_w(a_i), \ord_{w+\kappa}(a_i), \ldots , \ord_{w+(m-1)\kappa}(a_i)\bigr\}\Bigr\}\\
&=& \sum_{w\in\C}\max\bigl\{0, \ord_w(a_i)-\ord_{w+\kappa}(a_i), \ldots , \ord_w(a_i)-\ord_{w+(m-1)\kappa}(a_i)\bigr\}\\
&\leq & \sum_{w\in\C}\max\bigl\{0, \ord_w(a_i)-\ord_{w+\kappa}(a_i)\bigr\}+\cdots \\
& & \qquad +\sum_{w\in\C}\max\bigl\{0, \ord_w(a_i)-\ord_{w+(m-1)\kappa}(a_i)\bigr\}\\
&=& \sum_{w\in\C}\Bigl\{\ord_w(a_i)- \min\bigl\{\ord_w(a_i), \ord_{w+\kappa}(a_i)\bigr\}\Bigr\}+\cdots \\
& & \qquad +\sum_{w\in\C}\Bigl\{\ord_w(a_i) - \min\bigl\{\ord_w(a_i), \ord_{w+(m-1)\kappa}(a_i)\bigr\}\Bigr\}\\
&=&\tilde{n}_{\kappa}(a_i)+\cdots +\tilde{n}_{(m-1)\kappa}(a_i)=\sum_{j=1}^{m-1}\tilde{n}_{j\kappa}(a_i),
\end{eqnarray*}
which can be sharp when $m=2$ but the following example shows this is a crude estimate for our purposes.

\begin{example}\label{example_crude}
Given $c\in\mathbb{C}\setminus\{0\}$, we have the identity
$$
(z^2-c^2)\bigl\{(z+1)^2-c^2\bigr\}+(z^2+c^2)\bigl\{(z+1)^2+c^2\bigr\}=2z^2(z+1)^2+2c^4.
$$
Thus we have a solution $(a_1, a_2, a_3, a_4)$ to the equation
$$
a_1+a_2+a_3=a_4
$$
with $m=3$, such that $\{a_1, a_2, a_3\}$ is a linear independent system of relatively prime polynomials by requiring $c\neq 0, \pm 1, \pm \sqrt{-1}$.
In fact, we put
$$
p_1(z)=z^2-c^2, \quad p_2(z)=z^2+c^2, \quad p_3(z)=\sqrt{-2}\,z^2
$$
and
$$
a_1(z)=p_1(z)p_1(z+1), \quad a_2(z)=p_2(z)p_2(z+1), \quad a_3(z)=p_3(z)p_3(z+1),
$$
and $a_4(z)=2c^4$.
Then we have $\max_{1\leq i\leq 4} \{ \deg a_i \}=4$ and
$$
\sum_{i=1}^3 \deg\mathrm{r\tilde{a}d}_{1}(a_i)=2+2+2=6
$$
by observing the zeros of $a_i$ $(i=1,2,3)$, respectively.
Hence it is not possible for us to replace $\sum_{i=1}^{m+1}\tilde{n}_{\kappa}^{[m-1]}(a_i)$ in the above estimate by
$\sum_{i=1}^{m+1}\tilde{n}_{\kappa}(a_i)$, since $\frac{1}{2}m(m-1)=3$.
On the other hand, this example gives
$$
\sum_{i=1}^{4}\tilde{n}_{1}^{[2]}(a_i)=4+4+4=12,
$$
so that this is far from an example to confirm whether our estimate is sharp, unfortunately.
For this purpose, one needs to consider such an example that the $a_i(z)$ are of the form $a_i(z)=p_i(z)p_i(z+\kappa)\cdots p_i\bigl(z+(n-1)\kappa\bigr)$ for $n\geq m$ so that  $\min_{0\leq j \leq m-1}\bigl\{\ord_{w+j\kappa}(a_i)\bigr\}$ is positive at a zero of $a_i(z)$.
Note that this quantity is always zero when $n<m$ and the zeros of $p_i(z)$ appear non-periodically with respect to $\kappa$.
\end{example}

The following example observes the acuity of Theorem~\ref{differenceMason_m+1} in the case when $m=3$ with $\kappa=1$.

\begin{example}
Define
\begin{align*}
a_1(z)&=A(z+\alpha)(z+\alpha+1)(z+\alpha+2)^2,\\
a_2(z)&=B(z+\beta)(z+\beta+1)(z+\beta+2)^2,\\
a_3(z)&=-(A+B)z(z+1)(z+2)(z+3).
\end{align*}
By simple computations, we see that $a_1(z)+a_2(z)+a_3(z)$ reduces to a polynomial, say $a_4(z)$, of degree at most $1$ when $\alpha\ne-1/4$ and
\begin{equation}
\displaystyle\beta=\frac{2\alpha+1}{8\alpha-2},\quad B=-\frac{(4\alpha-1)^2\label{001}
}{3}A.
\end{equation}
Indeed, we have with \eqref{001}
\begin{align*}
a_1(z)+a_2(z)+a_3(z)&=\frac{A (8\alpha^2-4\alpha -1)\left(32 \alpha ^3-8 \alpha ^2+4 \alpha-1 \right)}{4 (1-4 \alpha )^2}z\\
&\quad +\frac{A (8\alpha^2-4\alpha -1) \left(32 \alpha ^4+160 \alpha ^3-8 \alpha ^2+8 \alpha-3\right)}{16 (1-4 \alpha )^2}.
\end{align*}
We can choose $\alpha$ so that $a_j(z)$, $j=1, 2, 3, 4$ are relatively prime.
Then $a_j$ satisfy \eqref{a1am}, and $\displaystyle\max_{1\leq j\leq4}\{\deg a_j\}=4$,
$\tilde{n}_1^{[2]}(a_1)=3$, $\tilde{n}_1^{[2]}(a_2)=3$, $\tilde{n}_1^{[2]}(a_3)=2$.
If $a_4(z)$ is non-constant, then $\tilde{n}_1^{[2]}(a_4)=1$ which gives $4\leq 6$, which is not enough to show the sharpness of Theorem~\ref{differenceMason_m+1} for the case $m=3$.
Next, we set $\displaystyle\alpha=\frac{i}{2\sqrt{2}}$ and $\displaystyle A=2-\sqrt{2}i$.
Then $\displaystyle\beta=-\frac{i}{2\sqrt{2}}$, $\displaystyle B=2+\sqrt{2}i$, and $a_4(z)$ reduces to a constant $\displaystyle -\frac{9}{16}$,
which gives a somewhat sharper estimate $4\leq 5$ for  Theorem~\ref{differenceMason_m+1} when $a_{m+1}$ is a constant.
\end{example}

\section{Polynomial solutions of Fermat type difference equations}\label{polyFermat_sec}

Factorial polynomial is defined as
    \begin{equation*}
    t^{\overline{n}} = t(t+1)\cdots (t+n-1).
    \end{equation*}
We extend this notation for the factorial of a polynomial $p$ in $\C[z]$ as 
    \begin{equation*}
    [p]_\kappa^{\overline{n}}=p(z)p(z+\kappa)\cdots p(z+(n-1)\kappa),
    \end{equation*}
where the shift $\kappa\in \C\setminus\{0\}$.

As a consequence of Theorem~\ref{differenceMason} we obtain the following result on the non-existence of polynomial solutions to a difference Fermat equation.

\begin{theorem}\label{diffFermatthm}
Let $\kappa\in\mathbb{C}\setminus\{0\}$, $n\in\mathbb{N}$ and $a, b, c\in \mathbb{C}[z]$, not all constant. If $[a]_{\kappa}^{\bar{n}}$, $[b]_{\kappa}^{\bar{n}}$ and $[c]_{\kappa}^{\bar{n}}$ are relatively prime and satisfy
\begin{equation}\label{xyz}
[a]_{\kappa}^{\bar{n}}+[b]_{\kappa}^{\bar{n}}=[c]_{\kappa}^{\bar{n}},
\end{equation}
then $n\leq 2$. If at least one of  $a$, $b$ and $c$ is constant, then $n=1$.
\end{theorem}

\begin{proof}
Suppose first that none of $a$, $b$ and $c$ is constant. If \eqref{xyz} holds, then by Theorem~\ref{differenceMason}, we have
    \begin{equation*}
    \begin{split}
    \deg [a]_\kappa^{\overline{n}} &\leq \max\{\deg [a]_\kappa^{\overline{n}},\deg [b]_\kappa^{\overline{n}},\deg [c]_\kappa^{\overline{n}}\} \\
    &\leq \tilde n_\kappa([a]_\kappa^{\overline{n}}) + \tilde n_\kappa([b]_\kappa^{\overline{n}}) + \tilde n_\kappa([c]_\kappa^{\overline{n}}) - 1 \\
    &\leq \deg a  + \deg b + \deg c - 1.
    \end{split}
    \end{equation*}
Since $\deg [a]_\kappa^{\overline{n}}=n\deg a$, it follows that
    \begin{equation}\label{xineq}
    n\deg a \leq \deg a + \deg b + \deg c - 1.
    \end{equation}
By repeating the same argument for $b$ and $c$ instead of $a$, we have
    \begin{equation}\label{yineq}
    n\deg b \leq \deg a + \deg b + \deg c - 1
    \end{equation}
and
    \begin{equation}\label{zineq}
    n\deg c \leq \deg a + \deg b + \deg c - 1.
    \end{equation}
By combining \eqref{xineq}, \eqref{yineq} and \eqref{zineq}, it follows that
     \begin{equation*}
     n(\deg a + \deg b + \deg c) \leq 3(\deg a + \deg b + \deg c) - 3,
     \end{equation*}
and so $n\leq 2$.

Assume now that at least one of $a$, $b$ and $c$ is constant. Then by \eqref{xyz} exactly one of them, say $c$, is constant. Then, by \eqref{xineq} and \eqref{yineq}, we have
     \begin{equation*}
     n(\deg a + \deg b) \leq 2(\deg a + \deg b) - 2,
     \end{equation*}
which implies that $n\leq 1$.
\end{proof}

The following example shows that the assertion of Theorem~\ref{diffFermatthm} is sharp.

\begin{example}\label{deg2example}
Choosing $\kappa=1$ and defining
    \begin{equation*}
    \begin{split}
    a(z) &= z^2,\\
    b(z) &= -\frac{i}{2}\left(\sqrt{2}z^2 + 2z -\sqrt{2}\right),\\
    c(z) &= -\frac{1}{2}\left(\sqrt{2}z^2 - 2z -\sqrt{2}\right),
    \end{split}
    \end{equation*}
it follows that $[a]_{1}^{\bar{2}}$, $[b]_{1}^{\bar{2}}$ and $[c]_{1}^{\bar{2}}$ have no common factors, and they satisfy \eqref{xyz} with $n=2$.
\end{example}

\noindent\textit{Remark. } The assertions of Lemma \ref{difference_lemma} and of Theorems \ref{differenceMason} and \ref{diffFermatthm} remain valid in a more general setting where the polynomials are in $k[z]$, where $k$ is any algebraically closed field.

\medskip

The following theorem extends Theorem~\ref{diffFermatthm} to equations with arbitrarily many terms.

\begin{theorem}\label{diffFermatthm2}
If $m\geq 2$ and $p_1, \ldots , p_{m+1}$ are non-constant polynomials in $\mathbb{C}[z]$ such that
$[p_1]_{\kappa}^{\bar{n}},
[p_2]_{\kappa}^{\bar{n}}, \ldots , [p_{m+1}]_{\kappa}^{\bar{n}}$
are relatively prime and satisfy
\begin{equation}\label{xyz2}
[p_1]_{\kappa}^{\bar{n}}+[p_2]_{\kappa}^{\bar{n}}+ \cdots +[p_{m}]_{\kappa}^{\bar{n}}= [p_{m+1}]_{\kappa}^{\bar{n}}
\end{equation}
for some $\kappa\in\mathbb{C}\setminus\{0\}$ and $n\in\mathbb{N}$, then
\begin{equation}\label{deg_upper_bound}
n\leq m^2-1-\frac{m(m-1)}{2\max_{1\leq i \leq m+1} \deg p_{i}}.
\end{equation}
\end{theorem}

\begin{proof}
We may assume, without loss of generality, that $[p_1]_{\kappa}^{\bar{n}},
[p_2]_{\kappa}^{\bar{n}}, \ldots , [p_{m}]_{\kappa}^{\bar{n}}$ are linearly independent. For otherwise we may eliminate some of the polynomials from \eqref{xyz2} to obtain a shorter equation, which is of the same form, but contains only linearly independent terms. Suppose first that $n\geq m$. By using Theorem \ref{differenceMason_m+1} we obtain
\begin{equation} \label{eqn:Thm3.4}
\begin{split}
n \cdot \max _{1\leq i \leq m+1} \deg p_{i} &\leq \sum_{i=1}^{m+1} \deg\mathrm{r\tilde{a}d}_{\kappa}^{[m-1]}([p_{i}]_{\kappa}^{\bar{n}})-\frac{1}{2}m(m-1).
\end{split}
\end{equation}
Further we have
\begin{eqnarray*}
\deg\mathrm{r\tilde{a}d}_{\kappa}^{[m-1]}([p_{i}]_{\kappa}^{\bar{n}})
&=& \sum_{w\in \C}\Bigl\{ \ord_w([p_i]_{\kappa}^{\bar{n}}) \\
&  &   -\min \bigl\{ \ord_w([p_i]_{\kappa}^{\bar{n}}), \ord_{w+\kappa}([p_i]_{\kappa}^{\bar{n}}), \ldots , \ord_{w+(m-1)\kappa}([p_i]_{\kappa}^{\bar{n}})\bigr\} \Bigr\}  \\
&=& \sum_{w\in \C}\Bigl\{ \ord_w(p_i)+\ord_{w+\kappa}(p_i)+\cdots+\ord_{w+(n-1)\kappa}(p_i) \\
& &   -\min \bigl\{ \ord_w(p_i)+\ord_{w+\kappa}(p_i)+\cdots+\ord_{w+(n-1)\kappa}(p_i),  \\
& &   \quad  \quad          \ord_{w+\kappa}(p_i)+\ord_{w+2\kappa}(p_i)+\cdots+\ord_{w+n\kappa}(p_i), \ldots , \\
& &   \quad \quad           \ord_{w+(m-1)\kappa}(p_i)+\ord_{w+m\kappa}(p_i)+\cdots+\ord_{w+(n+m-2)\kappa}(p_i)
                            \bigr\} \Bigr\}  \\
&\leq& \sum_{w\in \C}\bigl\{ \ord_w(p_i)+\ord_{w+\kappa}(p_i)+\cdots+\ord_{w+(m-2)\kappa}(p_i)\bigr\} \\
&\leq& (m-1)\deg p_i,
\end{eqnarray*}
since, when $m\leq n$, it follows
\begin{eqnarray*}
& & \hspace{-5ex} \min \bigl\{ \ord_w(p_i)+\ord_{w+\kappa}(p_i)+\cdots+\ord_{w+(n-1)\kappa}(p_i), \\
& &  \ord_{w+\kappa}(p_i)+\ord_{w+2\kappa}(p_i)+\cdots+\ord_{w+n\kappa}(p_i), \ldots , \\
& &  \ord_{w+(m-1)\kappa}(p_i)+\ord_{w+m\kappa}(p_i)+\cdots+\ord_{w+(n+m-2)\kappa}(p_i) \bigr\} \\
&\geq & \ord_{w+(m-1)\kappa}(p_i)+\ord_{w+m\kappa}(p_i)+\cdots+\ord_{w+(n-1)\kappa}(p_i).
\end{eqnarray*}
Therefore $(\ref{eqn:Thm3.4})$ gives
\begin{equation*}
n \cdot \max _{1\leq i \leq m+1} \deg p_{i} \leq
(m+1) (m-1) \max _{1\leq i \leq m+1} \deg p_{i} -\frac{1}{2}m(m-1),
\end{equation*}
which implies the assertion in the case $n\geq m$.

Assume now that $m> n$. Then $n\leq m-1$, and thus we have
\begin{equation*}
\deg\bigl(\mathrm{r\tilde{a}d}_{\kappa}^{[m-1]}([p_{i}]_{\kappa}^{\bar{n}})\bigr) \leq n\deg(p_i) \leq (m-1)\deg(p_i).
\end{equation*}
Therefore, by using Theorem \ref{differenceMason_m+1}, the assertion follows.
\end{proof}

\begin{example}
We consider the sharpness of the inequality \eqref{deg_upper_bound} in the case $m=2$. Let us first look at the case where the maximal degree of the polynomial solutions of \eqref{xyz2} is one. In this case it can be seen by a direct substitution of arbitrary linear polynomials into \eqref{xyz2} that such solutions are never relatively prime when $n=2$. If the maximal degree of the polynomial solutions is two, then by Theorem~\ref{diffFermatthm2} we have $n\leq 5/2$. In Example~\ref{deg2example} we have given a solution for the equation \eqref{xyz2} with $m=2$ and $n=2$, which is optimal in this case, since $n$ is an integer.
\end{example}

Theorem~\ref{diffFermatthm2} immediately implies an upper bound for $n$ in \eqref{xyz2}, which only depends on $m$ as follows.

\begin{corollary}\label{diffFermatcor2}
If the assumptions of Theorem~\ref{diffFermatthm2} are satisfied, then $n\leq m^2-2$.
\end{corollary}

Choosing $m=2$ in Corollary~\ref{diffFermatcor2} implies the first assertion of Theorem~\ref{diffFermatthm}, namely that $n\leq 2$.


The final two results of this section deal with another canonical form of difference Fermat equations.

\begin{theorem}\label{diffFermatthm3}
If $m\geq 2$ and $p_1, \ldots , p_{m}$ are non-constant polynomials in $\mathbb{C}[z]$ such that
$[p_1]_{\kappa}^{\bar{n}},
[p_2]_{\kappa}^{\bar{n}}, \ldots ,
[p_{m}]_{\kappa}^{\bar{n}}$
are relatively prime and satisfy
\begin{equation}\label{xyz3}
[p_1]_{\kappa}^{\bar{n}}+[p_2]_{\kappa}^{\bar{n}}+ \cdots +[p_{m}]_{\kappa}^{\bar{n}}=1
\end{equation}
for some $\kappa\in\mathbb{C}\setminus\{0\}$ and $n\in\mathbb{N}$, then
\begin{equation*}
n\leq m^2 - m -\frac{m(m-1)}{2\max_{1\leq i \leq m} \deg p_{i}}.
\end{equation*}
\end{theorem}

\begin{proof}
From Theorem \ref{differenceMason_m+1} we have
\begin{equation*}
n \cdot \max _{1\leq i \leq m} \deg p_{i} \leq \sum_{i=1}^{m}\deg\mathrm{r\tilde{a}d}_{\kappa}^{[m-1]}([p_i]_\kappa^{\bar{n}})-\frac{1}{2}m(m-1),
\end{equation*}
and so a similar discussion as in the proof of Theorem~\ref{diffFermatthm2} implies the assertion.
\end{proof}

\begin{corollary}\label{diffFermatcor3}
If the assumptions of Theorem~\ref{diffFermatthm3} are satisfied, then $n\leq m^2-m-1$.
\end{corollary}

If at least one of the polynomials in the equation \eqref{xyz2} is constant, then \eqref{xyz2} reduces into \eqref{xyz3}. In particular, when $m=2$, Corollary~\ref{diffFermatcor3} then implies the second assertion of Theorem~\ref{diffFermatthm}, namely that $n = 1$.

\section{Transcendental entire solutions of Fermat type difference equations}

In this section we extend the results obtained in Section~\ref{polyFermat_sec} for the case of entire solutions of hyper-order strictly less than one. The hyper-order of an entire function $g$ is defined as
    \begin{equation*}
    \rho_2(g)=\limsup_{r\to\infty}\frac{\log^+\log^+ T(r,g)}{\log r},
    \end{equation*}
where $T(r,g)$ is the Nevanlinna characteristic function of $g$. For $\kappa\in\C\setminus\{0\}$ we denote by $\mathcal{P}^1_\kappa$ the field of period $\kappa$ meromorphic functions of hyper-order strictly less than one.

In the case of hyper-order $\geq 1$, for an arbitrary integer $n\geq 2$ there exists a transcendental entire function $f(z)$ such that $[f]_{\kappa}^{\bar{n}}$ reduces to a constant. For example, consider $f(z)=\exp \bigl(\pi(z) \omega^{z/\kappa}\bigr)$ where $\pi(z)$ is a $\kappa$-periodic entire function of order $\rho(\geq 1)$ and $\omega\not=1$ is an $n$th root of unity.
Then we have $\rho_2(f)=\rho$ and
\begin{equation}\label{ozawa_eq}
[f]_{\kappa}^{\bar{n}}=\prod_{j=0}^{n-1}f(z+j\kappa)=\prod_{j=0}^{n-1} \exp \left(\pi(z) \omega^{z/\kappa} \omega^{j}\right)
=\exp \left(\pi(z) \omega^{z/\kappa}\sum_{j=0}^{n-1}\omega^j\right)\equiv e^0=1.
\end{equation}
Here we have applied an existence theorem of prime periodic entire functions by M. Ozawa \cite[{\sc Theorems} 1 and 2]{ozawa:77}, where he proved that for arbitrarily given $\kappa\neq 0$ and $\rho$ $(1\leq \rho\leq \infty)$, there exists a $\kappa$-periodic entire function $\pi(z)$ of order $\rho$. Examples of the type \eqref{ozawa_eq} are in stark contrast to the behavior of polynomials, and so we want to rule them out in this note.

\begin{proposition}\label{diffFermatthm_entire}
Let $m\geq 2$ and let $f_1, \ldots , f_{m}$ be non-constant entire functions such that $\rho_2(f_i)<1$ for all $i\in\{1, \ldots, m\}$, and such that $[f_1]_{\kappa}^{\bar{n}}, [f_2]_{\kappa}^{\bar{n}}, \ldots, [f_m]_{\kappa}^{\bar{n}}$ are linearly independent over $\mathcal{P}_{\kappa}^1$ and furthermore $[f_1]_{\kappa}^{\bar{n}}, [f_2]_{\kappa}^{\bar{n}}, \ldots, [f_m]_{\kappa}^{\bar{n}}$ and $[f_{m+1}]_{\kappa}^{\bar{n}}$ have no common zeros. If
\begin{equation*}\label{diffFermat_entire}
[f_1]_{\kappa}^{\bar{n}}+[f_2]_{\kappa}^{\bar{n}}+ \cdots +[f_{m}]_{\kappa}^{\bar{n}}= [f_{m+1}]_{\kappa}^{\bar{n}}
\end{equation*}
for some $\kappa\in\mathbb{C}\setminus\{0\}$ and $n\in\mathbb{N}$, then $n\leq m^2-1$.
\end{proposition}

In Proposition~\ref{diffFermatthm_entire} we have assumed that the entire functions $[f_1]_{\kappa}^{\bar{n}}, [f_2]_{\kappa}^{\bar{n}}, \ldots, [f_{m+1}]_{\kappa}^{\bar{n}}$ do not have common zeros in an analogy of the assumption of relative primeness of the polynomials in Theorem~\ref{diffFermatthm2}.

\begin{proposition}\label{diffFermatthm_entire2}
Let $m\geq 2$ and let $f_1, \ldots , f_{m}$ be non-constant entire functions such that $\rho_2(f_i)<1$ for all $i\in\{1,\ldots,m\}$, and such that
$[f_1]_{\kappa}^{\bar{n}},
[f_2]_{\kappa}^{\bar{n}}, \ldots ,
[f_{m}]_{\kappa}^{\bar{n}}$
are linearly independent over $\mathcal{P}_{\kappa}^1$. If
\begin{equation*}\label{diffFermat_entire2}
[f_1]_{\kappa}^{\bar{n}}+[f_2]_{\kappa}^{\bar{n}}+ \cdots +[f_{m}]_{\kappa}^{\bar{n}}= 1,
\end{equation*}
for some $\kappa\in\mathbb{C}\setminus\{0\}$ and $n\in\mathbb{N}$, then $n\leq m^2-m$.
\end{proposition}

Before we can prove Propositions~\ref{diffFermatthm_entire} and \ref{diffFermatthm_entire2}, we need to introduce tools to handle entire functions. In particular, we will consider an extension of the notion of difference radical for entire functions, and define the corresponding Nevanlinna counting functions.

The \textit{order} of a holomorphic curve $g:\C\to\P^n$ is defined by
    \begin{equation*}\label{order}
    \sigma(g)=\limsup_{r\to\infty}\frac{\log^+ T_g(r)}{\log
    r},
    \end{equation*}
where $\log^+x=\max\{0,\log x\}$ for all $x\ge0$, and
    \begin{equation*}\label{characteristic}
    T_g(r):=\int_0^{2\pi}u(re^{i\theta})\frac{d\theta}{2\pi}-u(0),
    \quad u(z)=\sup_{k\in \{0,\ldots,n\}} \log|g_k(z)|,
    \end{equation*}
is the Cartan characteristic function of $g$ with the reduced representation $g=[g_0:\cdots:g_n]$. Similarly,
the \textit{hyper-order} of $g:\C\to\P^n$ is
    \begin{equation*}\label{horder}
    \varsigma(g)=\limsup_{r\to\infty}\frac{\log^+\log^+ T_g(r)}{\log
    r}.
    \end{equation*}
The following lemma \cite[Lemma~8.3]{halburdkt:14TAMS} is a useful tool in dealing with shifts in characteristic and Nevanlinna counting functions.

\begin{lemma}[\cite{halburdkt:14TAMS}]\label{technical}
Let $T:[0,+\infty)\to[0,+\infty)$ be a non-decreasing continuous
function and let $s\in(0,\infty)$. If the hyper-order of $T$ is
strictly less than one, i.e.,
    \begin{equation*}\label{assu}
    \limsup_{r\to\infty}\frac{\log\log T(r)}{\log r}=\varsigma<1
    \end{equation*}
and $\delta\in(0,1-\varsigma)$, then
   \begin{equation*}\label{concl}
    T(r+s) = T(r)+ o\left(\frac{T(r)}{r^{\delta}}\right),
    \end{equation*}
where $r$ runs to infinity outside of a set of finite logarithmic
measure.
\end{lemma}

We denote by $\overline{D}(s,z_0)=\{z\in\C:|z-z_0|\leq s\}$ the closed disc of radius $s>0$ centred at $z_0\in\C$. We define, as in \cite{cherryy:01}, the \textit{order} $\ord_\zeta (f)$ of a meromorphic function $f$ at $\zeta\in\C$ as the unique $\mu\in\Z$ such that
    \begin{equation*}
    \lim_{z\to\zeta} \frac{f(z)}{(z-\zeta)^\mu}\in\C\setminus\{0\}.
    \end{equation*}
With this notation $\ord_\zeta (f)>0$ if and only if $f$ has a zero of order $\ord_\zeta (f)$ at $\zeta$, and $\ord_\zeta (f)<0$ if and only if $f$ has a pole of order $-\ord_\zeta (f)$ at $\zeta$. We also adopt the notation   $\ord^+_\zeta (f)=\max\{0,\ord_\zeta (f)\}$ and $\ord^-_\zeta (f)=\max\{0,-\ord_\zeta (f)\}$. Now, given $q\in\N$, we define
    \begin{equation*}
    \tilde n^{[q]}_\kappa\left(r,\frac{1}{f}\right) = \sum_{w\in\overline{D}(0,r)}\left(\ord^+_w (f) - \min_{0\leq j\leq q}\{\ord^+_{w+j\kappa} (f)\}\right)
    \end{equation*}
as a difference analogue of the truncated counting function for the zeros of $f$. The corresponding integrated counting function is defined in the usual way as
    \begin{equation}\label{diffTruncN}
    \widetilde N^{[q]}_\kappa\left(r,\frac{1}{f}\right) = \int_0^r\frac{\tilde n^{[q]}_\kappa(t,1/f)-\tilde n^{[q]}_\kappa(0,1/f)}{t}\,dt + \tilde n^{[q]}_\kappa(0,1/f)\log r.
    \end{equation}
Also, by defining
    \begin{equation*}
    \lambda_2(f)=\limsup_{r\to\infty}\frac{\log^+\log^+ N\left(r,\frac{1}{f}\right)}{\log r},
    \end{equation*}
it follows that $\lambda_2(f)\leq \rho_2(f)$. The following lemma demonstrates how the truncation works with the counting function \eqref{diffTruncN}.

\begin{lemma}\label{Truncation_lemma}
Let $f\not\equiv0$ be entire, let $\kappa\in\C\setminus\{0\}$ and let $n,q\in\N$. If $\lambda_2(f)<1$, then
    \begin{equation*}
    \widetilde N^{[q]}_\kappa\left(r,\frac{1}{[f]_\kappa^{\overline{n}}}\right) \leq q N\left(r,\frac{1}{f}\right) + o\left(N(r,1/f)\right)
    \end{equation*}
as $r\to\infty$ outside of an exceptional set of finite logarithmic measure.
\end{lemma}

\begin{proof}
Suppose first that $n>q$. Then by definition
    \begin{equation}\label{ord_calc}
    \begin{split}
    \widetilde n^{[q]}_\kappa\left(r,\frac{1}{[f]_\kappa^{\overline{n}}}\right) &=\sum_{w\in\overline{D}(0,r)}\left(\sum_{i=0}^{n-1} \ord^+_w f(z+i\kappa) - \min_{j\in\{0,\ldots,q\}}\left\{\sum_{i=0}^{n-1} \ord^+_{w+j\kappa} f(z+i\kappa)\right\}\right)\\
    &= \sum_{w\in\overline{D}(0,r)}\left(\sum_{i=0}^{n-1} \ord^+_{w+i\kappa} (f) - \min_{j\in\{0,\ldots,q\}}\left\{\sum_{i=0}^{n-1} \ord^+_{w+(i+j)\kappa} (f)\right\}\right).
    \end{split}
    \end{equation}
Each term in the minimum on the right hand side of \eqref{ord_calc} contains the sum $\sum_{i=q}^{n-1}\ord^+_{w+i\kappa}(f)$. To see this, we may write $\sum_{i=0}^{n-1}\ord^+_{w+(i+j)\kappa}(f)=\sum_{k=j}^{n+j-1}\ord^+_{w+k\kappa}(f)$, if necessary. Therefore, it follows by \eqref{ord_calc} that
    \begin{equation}\label{ord_calc2}
    \begin{split}
    \widetilde n^{[q]}_\kappa\left(r,\frac{1}{[f]_\kappa^{\overline{n}}}\right) &\leq \sum_{w\in\overline{D}(0,r)}\left(\sum_{i=0}^{n-1} \ord^+_{w+i\kappa} (f) - \sum_{i=q}^{n-1} \ord^+_{w+i\kappa} (f) \right) \\
    &= \sum_{w\in\overline{D}(0,r)}\left(\sum_{i=0}^{q-1} \ord^+_{w+i\kappa} (f) \right) \\
    &= \sum_{i=0}^{q-1} n\left(r,\frac{1}{f(z+i\kappa)}\right).
    \end{split}
    \end{equation}
By integrating \eqref{ord_calc2} it follows that
    \begin{equation}\label{ord_calc3}
    \widetilde N^{[q]}_\kappa\left(r,\frac{1}{[f]_\kappa^{\overline{n}}}\right) \leq \sum_{i=0}^{q-1} N\left(r,\frac{1}{f(z+i\kappa)}\right),
    \end{equation}
since $n>q$. If $n\leq q$, the inequality \eqref{ord_calc3} holds trivially, so in fact we have \eqref{ord_calc3} for all $n\in\N$. The assertion now follows by Lemma~\ref{technical}.
\end{proof}


The following result is a truncated second main theorem for differences.

\begin{theorem}\label{trunc_diff_2mt}
Let $g_1,\ldots,g_m$ be $m\geq 2$ entire functions with no common zeros, linearly independent over $\mathcal{P}^1_\kappa$, and let $g_{m+1}=g_1 + \cdots + g_m$. If the holomorphic curve $g=[g_1:\cdots:g_m]$ satisfies $\varsigma(g)<1$, then
    \begin{equation}\label{2ndMT}
    T_g(r) \leq \sum_{j=1}^{m+1} \widetilde N^{[m-1]}_\kappa\left(r,\frac{1}{g_j}\right) + o\left(\frac{T_g(r)}{r^{1-\varsigma(g)-\varepsilon}}\right),
    \end{equation}
where $\kappa\in\C\setminus\{0\}$, $\varepsilon>0$, and $r\to\infty$ outside of an exceptional set of finite logarithmic measure.
\end{theorem}

\begin{proof}
Denote by $\mathcal{C}_\kappa(g_1 \dots g_m)$ the Casoratian of $g_1, \dots, g_m$, and define
$$
G=g_1 \cdots g_{m+1}/\mathcal{C}_\kappa(g_1 \dots g_m).
$$
Suppose $w$ is a zero of $G$. We assert that
\begin{equation}
\ord_w^+(G)\leq \sum_{j=1}^{m+1}\left(\ord_w^+(g_j)-\min_{i\in \{0,\dots, m-1\}} \{\ord^+_{w+i\kappa} (g_j)\}\right).\label{0.1}
\end{equation}
To confirm this, we write
$$
\frac{1}{G}= \frac{\epsilon}{g_{l_0}} \cdot
\left| \begin{array}{cccc}
1 & 1 & \ldots & 1 \\
\frac{g_{l_1}(z+\kappa)}{g_{l_1}(z)} & \frac{g_{l_2}(z+\kappa)}{g_{l_2}(z)} & \ldots & \frac{g_{l_m}(z+\kappa)}{g_{l_m}(z)}\\
\vdots & \vdots & \ddots & \vdots \\
\frac{g_{l_1}(z+(m-1)\kappa)}{g_{l_1}(z)} & \frac{g_{l_2}(z+(m-1)\kappa)}{g_{l_2}(z)} & \ldots & \frac{g_{l_m}(z+(m-1)\kappa)}{g_{l_m}(z)}
\end{array} \right|,
$$
where $\epsilon \in \{\pm 1\}$, and the indexes $\{l_1,\ldots,\l_m\}\subset\{1,\ldots,m+1\}$ and $l_0\in\{1,\ldots,m+1\}\setminus\{l_1,\ldots,\l_m\}$ depend on $z$ so that $g_{l_0}(z)\not=0$ for all $z\in\C$. We see that the zeros of $G$ are the poles of some of $g_{l_j}(z+i\kappa)/g_{l_j}(z)$, $1\leq j\leq m$, $1\leq i\leq m-1$. The maximal order of pole among the $j$th column in the determinant above is given by
$$
\max_{i\in \{0,\dots, m-1\}}\left(\ord_w^+(g_{l_j})- \{\ord^+_{w+i\kappa} (g_{l_j})\}\right)=\ord_w^+(g_{l_j})-\min_{i\in \{0,\dots, m-1\}} \{\ord^+_{w+i\kappa} (g_{l_j})\}.
$$
Since $g_{l_0}(z)\not=0$ for all $z\in\C$, we obtain \eqref{0.1}, and hence we have
\begin{equation}\label{nonint}
n\left(r,\frac{1}{G}\right)\leq \sum_{j=1}^{m+1}\tilde n_\kappa^{[m-1]}\left(r,\frac1{g_j}\right).
\end{equation}
The assertion now follows by integrating \eqref{nonint} and applying \cite[Theorem~2.1]{halburdkt:14TAMS}.
\end{proof}

%
%


\subsection{Proof of Proposition~\ref{diffFermatthm_entire}}

%
By denoting
    \begin{equation}\label{Nr}
    N(r)=\sup_{j\in\{1,\ldots,m+1\}} N\left(r,\frac{1}{g_j}\right),
    \end{equation}
where $g_j=[f_j]_\kappa^{\overline{n}}$, $j=1,\ldots,m+1$, and applying \eqref{2ndMT} we have
    \begin{equation*}
    N(r) \leq \sum_{j=1}^{m+1} \widetilde N^{[m-1]}_\kappa\left(r,\frac{1}{[f_j]_\kappa^{\overline{n}}}\right)  + o\left(T_g(r)\right)
    \end{equation*}
as $r\to\infty$ outside of an exceptional set $E$ of finite logarithmic measure. By defining
    \begin{equation}\label{Mr}
    M(r)=\sup_{j\in\{1,\ldots,m+1\}} N\left(r,\frac{1}{f_j}\right),
    \end{equation}
we have
    \begin{equation*}
    nM(r-n|\kappa|) \leq N(r) \leq n M(r+n|\kappa|)
    \end{equation*}
for all $r\geq n|\kappa|$, and so by Lemma~\ref{technical} it follows that
    \begin{equation*}
    N(r) = nM(r) + o(M(r))
    \end{equation*}
as $r\to\infty$ outside of an exceptional set $F$ of finite logarithmic measure. Therefore, Lemma~\ref{Truncation_lemma} yields
    \begin{equation*}
    \begin{split}
    n M(r) &= N(r) +  o\left(M(r)\right) \\
    &\leq \sum_{j=1}^{m+1}(m-1) N \left(r,\frac{1}{f_j}\right)  + o\left(M(r)\right) \\
    &\leq (m+1)(m-1)M(r)+ o\left(M(r)\right)
    \end{split}
    \end{equation*}
as $r\to\infty$ outside of $E\cup F$, and so $n\leq m^2-1$. \hfill$\Box$

\subsection{Proof of Proposition \ref{diffFermatthm_entire2}}

As in the proof of Proposition~\ref{diffFermatthm_entire}, we apply \eqref{2ndMT}, but now with $g_j=[f_j]_\kappa^{\overline{n}}$, $j=1,\ldots,m$. Note that then \eqref{Nr} reduces into
    \begin{equation*}
    N(r)=\sup_{j\in\{1,\ldots,m\}} N\left(r,\frac{1}{g_j}\right) + O(1),
    \end{equation*}
and we have
    \begin{equation*}
    N(r) \leq \sum_{j=1}^{m} \widetilde N^{[m-1]}_\kappa\left(r,\frac{1}{[f_j]_\kappa^{\overline{n}}}\right)  + o\left(T_g(r)\right)
    \end{equation*}
as $r\to\infty$ outside of an exceptional set $E$ of finite logarithmic measure. Similarly, \eqref{Mr} simplifies to
    \begin{equation*}
    M(r)=\sup_{j\in\{1,\ldots,m\}} N\left(r,\frac{1}{f_j}\right)  + O(1),
    \end{equation*}
and so
    \begin{equation*}
    nM(r-n|\kappa|) \leq N(r) \leq n M(r+n|\kappa|)
    \end{equation*}
for all $r\geq n|\kappa|$. Now Lemma~\ref{technical} yields
    \begin{equation*}
    N(r) = nM(r) + o(M(r))
    \end{equation*}
as $r\to\infty$ outside of an exceptional set $F$ of finite logarithmic measure. Therefore, by Lemma~\ref{Truncation_lemma}, we have
    \begin{equation*}
    \begin{split}
    n M(r) &= N(r) +  o\left(M(r)\right) \\
    &\leq \sum_{j=1}^{m}(m-1) N \left(r,\frac{1}{f_j}\right)  + o\left(M(r)\right) \\
    &\leq m(m-1)M(r)+ o\left(M(r)\right)
    \end{split}
    \end{equation*}
as $r\to\infty$ outside of $E\cup F$, and so $n\leq m^2-m$. \hfill$\Box$

%

\def\cprime{$'$}
\providecommand{\bysame}{\leavevmode\hbox to3em{\hrulefill}\thinspace}
\providecommand{\MR}{\relax\ifhmode\unskip\space\fi MR }
\providecommand{\MRhref}[2]{%
  \href{http://www.ams.org/mathscinet-getitem?mr=#1}{#2}
}
\providecommand{\href}[2]{#2}


\begin{thebibliography}{10}

\bibitem{cherryy:01}
W.~Cherry and Z.~Ye, \emph{{N}evanlinna's theory of value distribution},
  Springer-Verlag, Berlin, 2001. \MR{1831783}

\bibitem{dyakonov:12}
K.~M. Dyakonov, \emph{Zeros of analytic functions, with or without
  multiplicities}, Math. Ann. \textbf{352} (2012), no.~3, 625--641.
  \MR{2885590}

\bibitem{gundersenh:04}
G.~G. Gundersen and W.~K. Hayman, \emph{The strength of {C}artan's version of
  {N}evanlinna theory}, Bull. London Math. Soc. \textbf{36} (2004), no.~4,
  433--454. \MR{2069006}

\bibitem{halburdkt:14TAMS}
R.~G. Halburd, R.~J. Korhonen, and K.~Tohge, \emph{Holomorphic curves with
  shift-invariant hyperplane preimages}, Trans. Amer. Math. Soc. \textbf{366}
  (2014), no.~8, 4267--4298. \MR{3206459}

\bibitem{hayman:84}
W.~K. Hayman, \emph{Warings {P}roblem f\"ur analytische {F}unktionen}, Bayer.
  Akad. Wiss. Math.-Natur. Kl. Sitzungsber. (1984), 1--13 (1985). \MR{0803374}

\bibitem{hayman:14}
W.~K. Hayman, \emph{Waring's theorem and the super {F}ermat problem for numbers and functions}, Complex Var. Elliptic Equ. \textbf{59} (2014), no.~1, 85--90. \MR{3170744}

\bibitem{huy:02}
P.-C. Hu and C.-C. Yang, \emph{A note on the {\it abc} conjecture}, Comm. Pure
  Appl. Math. \textbf{55} (2002), no.~9, 1089--1103. \MR{1908663}

\bibitem{laeng:99}
E.~Laeng, \emph{Fermat's last theorem for polynomials}, Parabola \textbf{35}
  (1999), no.~1, 1--5.

\bibitem{lang:90}
S.~Lang, \emph{Old and new conjectured {D}iophantine inequalities}, Bull. Amer.
  Math. Soc. (N.S.) \textbf{23} (1990), no.~1, 37--75. \MR{1005184}

\bibitem{li:15}
N.~Li, \emph{On the existence of solutions of a {F}ermat-type difference
  equation}, Ann. Acad. Sci. Fenn. Math. \textbf{40} (2015), no.~2, 907--921.
  \MR{3409710}

\bibitem{mason:84}
R.~C. Mason, \emph{Diophantine equations over function fields}, London
  Mathematical Society Lecture Note Series, vol.~96, Cambridge University
  Press, Cambridge, 1984. \MR{754559}

\bibitem{ozawa:77}
M.~Ozawa, \emph{On the existence of prime periodic entire functions}, K\=odai
  Math. Sem. Rep. \textbf{29} (1977/78), no.~3, 308--321. \MR{510300}

\bibitem{snyder:00}
N.~Snyder, \emph{An alternate proof of {M}ason's theorem}, Elem. Math.
  \textbf{55} (2000), no.~3, 93--94. \MR{1781918}

\bibitem{stothers:81}
W.~W. Stothers, \emph{Polynomial identities and {H}auptmoduln}, Quart. J. Math.
  Oxford Ser. (2) \textbf{32} (1981), no.~127, 349--370. \MR{625647}

\end{thebibliography}

\end{document}